\documentclass[a4paper,12pt,reqno]{amsart}

\pdfoutput=1

\usepackage[margin=2.5cm, headsep=1cm, footskip=1cm]{geometry}
\usepackage[hang, flushmargin]{footmisc}
\usepackage[english]{isodate}
\usepackage{amsmath, amsthm, amssymb, mathtools}
\usepackage{enumitem}
\usepackage{xcolor}
\usepackage{tikz, tikzsymbols, tikz-cd}
\usepackage{float}
\restylefloat{figure}
 \usepackage[cmtip,all]{xy}
\usepackage{hyperref}
\usepackage{multicol}
\usepackage[capitalise, nameinlink, noabbrev, nosort]{cleveref}

\definecolor{indigo}{HTML}{492DA5}

\allowdisplaybreaks[2]
\setenumerate{listparindent=\parindent}

\providecommand{\noopsort}[1]{}

\makeatletter\g@addto@macro\bfseries{\boldmath}\makeatother

\makeatletter
\let\origsection\section
\renewcommand\section{\@ifstar{\starsection}{\nostarsection}}
\newcommand\sectionspace{\vspace{0.5ex}}
\newcommand\nostarsection[1]{\sectionspace\origsection{#1}\sectionspace}
\newcommand\starsection[1]{\sectionspace\origsection*{#1}\sectionspace}
\makeatother

\crefname{page}{page}{pages}

\setlist[enumerate]{font=\normalfont}
\crefname{enumi}{}{}
\crefname{enumii}{}{}

\numberwithin{equation}{section}
\crefname{equation}{equation}{equations}

\crefname{condition}{condition}{conditions}
\creflabelformat{condition}{#2(#1)#3}

\newtheorem{theorem}{Theorem}[section]

\newtheorem{thm}[theorem]{Theorem}
\crefname{thm}{Theorem}{Theorems}

\newtheorem{lemma}[theorem]{Lemma}
\crefname{lemma}{Lemma}{Lemmas}

\crefname{prop}{Proposition}{Propositions}

\newtheorem{cor}[theorem]{Corollary}
\crefname{cor}{Corollary}{Corollaries}

\theoremstyle{definition}

\crefname{definition}{Definition}{Definitions}

\theoremstyle{remark}

\newtheorem{remark}[theorem]{Remark}
\crefname{remark}{Remark}{Remarks}

\crefname{remarks}{Remarks}{Remarks}

\crefname{example}{Example}{Examples}

\newcommand{\vecspan}{\operatorname{span}}

\begin{document}

\title{Uniqueness Theorems for Twisted Steinberg Algebras}

\author{Rizalyn Bongcawel}
\address[R.S.~Bongcawel]{Department of Mathematics and Statistics,Mindanao State University-Iligan Institute of Technology,
Tibanga, Iligan City/ Jose Rizal Memorial State University-Dipolog Campus, Dipolog City, THE PHILIPPINES}
\email{rizalyn.bongcawel@g.msuiit.edu.ph}

\author{Lyster Rey Cabardo}
\address[L.R.B.~Cabardo]{Department of Mathematics and Statistics, College of Science and Mathematics, and Center for Mathematical and Theoretical Physical Sciences, Premier Research Institute of Science and Mathematics, Mindanao State University-Iligan Institute of Technology,
Tibanga, Iligan City, THE PHILIPPINES}
\email{lysterrey.cabardo@g.msuiit.edu.ph}

\author{Lisa Orloff Clark}
\address[L.O.~Clark]{School of Mathematics and Statistics, Victoria University of Wellington, PO Box 600, Wellington 6140, NEW ZEALAND}
\email{lisa.orloffclark@vuw.ac.nz}


\begin{abstract}
Given an ample Hausdorff groupoid $G$, a unital commutative ring $R$, and a discrete twist $(\Sigma,i,q)$, we establish a generalised uniqueness theorem for the twisted Steinberg algebra $A_R(G;\Sigma)$.  By applying this theorem when $G$ is effective, we establish a  Cuntz-Krieger uniqueness theorem as a corollary. We also prove a generalised graded uniqueness theorem for $A_R(G;\Sigma)$.
\end{abstract}

\maketitle 

\section{Introduction}

Uniqueness theorems give checkable conditions under which a homomorphism between rings/algebras is injective.  There are numerous uniqueness theorem results in analysis for C*-algebras, see for example  \cite{Armstrong2022,BCFS2014,BNR,Fl,GLR,HKLQ,Kak,Kang,KPRR1997,RSY, Szy2002}. These give inspiration to analogous theorems in ring theory, including for Leavitt path algebras associated to directed graphs \cite{LPA2005,TOMFORDE2007270}, their higher-rank analogues Kumjian–Pask algebras in \cite{ACaHR, CaHF, CLARK2017364}, and more general Steinberg algebras arising from suitable groupoids in \cite{CE2015, CFST2014, GUTSA}.
Such results are particularly useful when identifying one class of algebras as a subclass of another. For instance, the fact that every Leavitt path algebra can be realised as a Steinberg algebra follows from the graded uniqueness theorem, see for example \cite[Example~3.2]{CS2015}.

In this paper, we work in the setting of twisted Steinberg algebras \cite{ACCCLMRSS,ACCLMR} and establish both a generalised uniqueness theorem and graded uniqueness theorem. Our results extend and unify existing cocycle-based uniqueness theorems, which have previously only been considered under an effectivity assumption on the underlying groupoid \cite[Theorem~6.1 and Theorem~7.2]{ACCLMR}. By removing this restriction and allowing more general twists, we broaden the range of algebras to which these techniques apply. In particular, our results recover all the aforementioned untwisted uniqueness theorems in ring theory, as well as their graded analogues, as special cases. 

Our approach combines techniques developed in these special cases with ideas from the C*-algebraic setting, as in \cite{Armstrong2022}. Adapting these arguments to the purely algebraic context requires care, since the analytic approach relies on states and the disintegration of representations, which are tools that are not available in our setting. 

In \cref{sec2}, we give the necessary preliminaries on the construction of twisted Steinberg algebras as in \cite{ACCCLMRSS, ACCLMR} and graded twisted Steinberg algebras as in \cite{naingue}. In \cref{sec3}, we established the generalised and Cuntz-Krieger uniqueness theorem in \cref{CK1}. In \cref{sec4}, we prove a generalised graded uniqueness theorem and present a corollary where we add the assumption that the underlying groupoid corresponding to the homogeneous component with degrees $\epsilon$ is effective.

\section{Preliminaries}
\label{sec2}
In this section, we recall how a Steinberg algebra is built from an ample groupoid as independently introduced in \cite{CS2015} and \cite{Steinberg2010}.  We then consider more general twisted Steinberg algebras as in \cite{ACCLMR} and \cite{ACCCLMRSS}, and graded twisted Steinberg algebras as in \cite{naingue}.

\bigskip

\noindent \textbf{Ample groupoids}:  A \textit{groupoid} $G$ is a small category in which every morphism is invertible. A detailed definition is found, for example, in \cite{Sims2020}.  We denote the \textit{composable pairs}  in a groupoid $G$ by $G^{(2)}\subseteq G \times G$, and the \textit{range} and \textit{source} maps by $r,s:G \to G$ where  $s(\alpha)=\alpha^{-1}\alpha$ and  $r(\alpha)=\alpha\alpha^{-1}$, for all $\alpha \in G$.  We write  $G^{(0)} := s(G) = r(G)$ for the \textit{unit space} of $G$. A \textit{topological groupoid} is a groupoid $G$ that is also a topological space and the topology is compatible with its structure, that is, inversion and composition are continuous maps. We say the groupoid is \textit{\'etale} if the source (or equivalently the range) map is a local homeomorphism. An open subset $U$ of $G$ is called an \textit{open bisection} if $s|_U$ and $r|_U$ are homeomorphisms onto open subsets of $G^{(0)}$. An \'etale groupoid $G$ is \textit{ample} if it has a basis of compact open bisections. 

Denote $G_x = r^{-1}(x)$, $G^x = s^{-1}(x)$, and $G^y_x = {G_x \cap G^y}$. The \textit{isotropy} of a groupoid $G$ is the set $$\mathrm{Iso}(G):= \{ \gamma \in G: r(\gamma)= s(\gamma)\} = \bigcup\limits_{x\in G^{(0)}} {G_x^x}.$$ If the interior of the isotropy, $\text{Iso}(G)^\circ$, is equal to $G^{(0)}$, then we call $G$ \textit{effective}. 

\medskip

 \noindent \textbf{Steinberg algebras}:  Let $G$ be an ample Hausdorff groupoid  and $R$ be  a unital commutative ring. Given a topological space 
$X$ and a ring $R$, the \textit{open support} of a function $f:X \to R$ is the set 
 $$\mathrm{supp}(f):=\{x \in X: f(x) \neq 0\} = f^{-1}(R\backslash \{0\}).$$ We say that 
 $f$ is \textit{compactly supported} if ${\mathrm{supp}(f)}$ is contained in a compact set. 
The \textit{Steinberg algebra} $A_R(G)$ of $G$ over $R$  is the set of all locally constant, compactly supported  functions from $G$ to $R$.  It is generated by the characteristic functions, denoted $1_U$ where $1_U:G \to R$ and takes value 1 on $U$ and 0 otherwise,  so that 
$$A_R(G)=\vecspan\{1_U |~ U~\text{is a compact open bisection of }G\}.$$  \textit{Convolution} is defined such that for $f,g \in A_R(G)$ and $\gamma \in G$ we have  \begin{equation*}
 	(f*g)(\gamma):=  \sum\limits_{\substack {(\alpha,\beta) \in G^{(2)},\\ \alpha \beta =\gamma}} f(\alpha)g(\beta)= \sum_{\eta \in G^{r(\gamma)}} f(\gamma \eta)g(\eta^{-1}).
 \end{equation*}

 \noindent \textbf{Twisted convolution and discrete twists}:  In \cite{ACCLMR} and \cite{ACCCLMRSS}, a Steinberg algebra is generalised in two ways.  First, using a continuous $2$-cocyle $\sigma: G^{(2)} \to T \leq R^\times$ to ``twist'' the convolution. This produces a twisted Steinberg algebra $A_R(G;\sigma)$ which is identical to the untwisted one as an $R$-module but with convolution product defined such that for $f,g \in A_R(G)$ and $x \in G$ we have \begin{equation*}
		(f *_\sigma g)(\gamma) := \sum\limits_{\substack {(\alpha,\beta) \in G^{(2)},\\ \alpha \beta =\gamma}} \sigma(\alpha,\beta)f(\alpha)g(\beta) = \sum_{\eta \in G^{r(\gamma)}} \sigma(\gamma\eta,\eta^{-1}) f(\gamma\eta)g(\eta^{-1}),
	\end{equation*} for every $\gamma \in G$.

    Second, the groupoid itself is  ``twisted". For this we define a \textit{discrete twist} by $R^{\times}$, the group of multiplicative units of $R$, over an ample Hausdorff groupoid $G$ as a sequence 
		$G^{(0)} \times R^{\times} \overset{i}{\hookrightarrow} \Sigma \overset{q}{\hookrightarrow} G$,
	where the groupoid $G^{(0)} \times R^{\times}$ is regarded as trivial group bundle with fibres $R^{\times}$, $\Sigma$ is a Hausdorff \'etale groupoid with $\Sigma^{(0)} = i(G^{(0)} \times \{1\})$, and $i$ and $q$ are continuous groupoid homomorphisms that restrict to homeomorphisms of unit spaces. Furthermore the following conditions are satisfied:
    \begin{enumerate}
		\item[(DT1)] The sequence is \textit{exact}, in the sense that $i(\{x\} \times R^{\times})=q^{-1}(x)$ for every $x\in G^{(0)}$, $i$ is injective, and $q$ is surjective.
		\item[(DT2)] The groupoid $\Sigma$ is a \textit{locally trivial $G$-bundle}, in the sense that for each $\alpha \in G$, there is an open bisection $B_{\alpha}$ of $G$ containing $\alpha$, and a continuous map $P_{\alpha}: B_{\alpha} \to \Sigma$ such that 
		\begin{itemize}
			\item[(i)] $q \circ P_{\alpha} = \mathrm{id}_{B_{\alpha}}$;
			\item[(ii)] the map $(\beta, t) \mapsto i(r(\beta,t))P_{\alpha}(\beta)$ is a homeomorphism from $B_{\alpha}\times R^{\times}$ to $q^{-1}(B_{\alpha})$. 
		\end{itemize}   
		\item[(DT3)] The image of $i$ is \textit{central} in $\Sigma$, in the sense that for all $\sigma \in \Sigma$ and $t \in R^{\times}$, $i(r(\sigma), t)\sigma = \sigma~i(s(\sigma), t)$ .  
	\end{enumerate}
    We will denote a discrete twist over $G$ by $(\Sigma,i,q)$) or simply $\Sigma$. If $G$ is ample, then $\Sigma$ is ample. 
Given a discrete twist $\Sigma$, there is a continuous free action of $R^{\times}$ on $\Sigma$ (see \cite{ACCLMR})  given by $$t \cdot \sigma = i(r(\sigma),t)\sigma~\text{for all}~t \in R^{\times}~\text{and}~\sigma \in \Sigma.$$ 

    \noindent \textbf{Twisted Steinberg algebras}:  Given a topological space $X$ and a ring $R$, the $R$-module of locally constant functions from $X$ to $R$ is written as $C(X,R)$.  For every $f \in C(X,R)$, $\text{supp}(f)$ is a clopen set and we say $f$ is \textit{$R^\times$-contravariant} if $f(t \cdot \sigma)=t^{-1}f(\sigma)$.  The twisted Steinberg algebra of a discrete twist $(\Sigma, i, q)$ over $G$ is the set $$A_R(G;\Sigma) := \{ f \in C(\Sigma,R): f~\text{is}~ R^\times~\text{contravariant and}~q(\text{supp}(f))~\text{is compact}\}.$$ Convolution is defined using a section $S:G \to \Sigma$, not necessarily continuous,  and is independent of the choice of section.  More precisely, if $S$ is a section, $f,g \in A_R(G;\Sigma)$ and $\sigma \in \Sigma$, convolution is given by
    $$(f * g)(\sigma) := \sum\limits_{\alpha\in G^{r(\sigma)}\cap \text{supp}_G(f)} f(S(\alpha))g(S(\alpha)^{-1}\sigma).$$
    
Similar to the untwisted Steinberg algebra, in a sense, the twisted Steinberg algebra is generated by functions on compact open bisections.  For a compact open bisection $X \subseteq \Sigma$, define $\tilde{1}_X:\Sigma \to R$ such that for $\sigma \in \Sigma$ 
\[
\tilde{1}_{X}(\sigma) = 
\begin{cases}
t^{-1}, & \text{if}~\sigma \in tX,~\text{for some}~t \in R^\times,\\
0, & \text{otherwise}.
\end{cases}
\]
Then by \cite[Proposition~2.8]{ACCCLMRSS}
\[A_R(G;\Sigma) = \vecspan \{\tilde{1}_{X} \mid X \text{ is a compact open bisection of } \Sigma\}.\]

\bigskip

\noindent \textbf{Graded twists and graded twisted Steinberg algebras}: Let  $\Gamma$ be a discrete group with identity $\epsilon$. A \textit{$\Gamma$-graded discrete twist} is a discrete twist over $G$ given by the sequence $G^{(0)} \times R^{\times} \overset{i}{\hookrightarrow} \Sigma \overset{q}{\hookrightarrow} G$  together with continuous groupoid homomorphisms $c_\Sigma: \Sigma \to \Gamma$ and $c_G: G \to \Gamma$ such that the following diagram commutes \[
	\xymatrix@1{
		G^{(0)} \times T \ar[r]^-{i} ~ &~ \Sigma \ar[r]^{q} \ar[dr]_{c_{\Sigma}}  & G \ar[d]^{c_{G}~~.} \\
		&& \Gamma
	}
	\]
The maps $c_G$ and $c_\Sigma$ are continuous $\Gamma$-gradings on $G$ and $\Sigma$, respectively. Hence, $G_\gamma \coloneqq c_G^{-1}(\gamma)$ and $\Sigma_\gamma \coloneqq c_\Sigma^{-1}(\gamma)$ are clopen sets and $G_\epsilon$ is a clopen subgroupoid of $G$. It is shown in \cite[Proposition~4.2]{naingue} that graded discrete twists give rise to a grading on the twisted Steinberg algebra, more specifically $A_R(G;\Sigma)$ is a $\Gamma$-graded $R$-algebra with homogeneous components given by $A_R(G;\Sigma)_\gamma :=\{f \in A_R(G;\Sigma):~\text{supp}_G(f) \subseteq G_\gamma\}$.

\section{A generalised Uniqueness Theorem}
\label{sec3}
In this section, we prove a generalised uniqueness theorem for twisted Steinberg algebras. This improves on  the general uniqueness theorem for untwisted Steinberg algebras \cite[Theorem~3.1]{GUTSA}.  We then prove a Cuntz-Krieger uniqueness theorem as a special case.  Since twisted Steinberg algebras $A_R(G;\Sigma)$ generalise Steinberg algebras $A_R(G;\sigma)$ twisted by a 2-cocycle $\sigma$, our theorem gives a generalised uniqueness theorem for these as well.  We also recover the Cuntz-Kreiger uniqueness theorem  for cocycle twisted algebras $A_R(G;\sigma)$, which is \cite[Theorem~6.1]{ACCLMR}.
The following lemma is important and widely used to prove results in this paper. 
\begin{lemma}\label{lemma1}
     Let $G$ be an ample Hausdorff groupoid, let $R$ be a unital commutative ring, let $(\Sigma,i,q)$ be a discrete twist over $G$ and let $\mathcal{F}$ be a finite collection of compact open bisections of $\Sigma$ with mutually disjoint images in $G$. For any $D_1, D_2 \in \mathcal{F}$ with $D_1 \not= D_2$, $D_2^{-1}D_1 \cap i(G^{(0)} \times R^{\times}) = \varnothing.$
\end{lemma}

\begin{proof}  Suppose, by way of contradiction, that there exists $\gamma \in D_2^{-1}D_1 \cap i(G^{(0)} \times R^{\times})$.
    Then it can be written as $\gamma = \alpha^{-1}\beta$ where $\alpha \in D_2$ and $\beta \in D_1$. Since $\gamma \in i(G^{(0)} \times R^{\times})$,
  \[q(\gamma) = q(\alpha^{-1}\beta) \in G^{(0)}.\]  
    Since $q$ is a homomorphism, we have $q(\beta) = q(\alpha)$. But elements of $\mathcal{F}$ have mutually disjoint images in $G$; consequently $D_1=D_2$, which is a contradiction.
\end{proof}
For our generalised uniqueness theorem, we show that injectivity of a homomorphism on a twisted Steinberg algebra can be determined by checking injectivity only on functions supported on the interior of the isotropy. The next lemma is \cite[Lemma~3.6]{BCF} which we use in \cref{cor:intisotwist} to show that the interior of the isotropy of $G$ is actually a twist over the interior of the isotropy of $\Sigma$.

\begin{lemma}\cite[Lemma~3.6]{BCF}\label{lem:BCF}
    Let $R$ be a commutative ring with identity and $(\Sigma,i,q)$ be a discrete twist over an ample Hausdorff groupoid $G$. Let $H$ be an open subgroupoid of $G$ and let $\Sigma|_H=q^{-1}(H)$. Then, restricting the maps $i$ and $q$, we get a discrete twist $$H^{(0)} \times R^\times \to \Sigma|_H\to H$$ and $\iota_H:A_R(H;\Sigma|_H)\to A_R(G;\Sigma)$ a natural inclusion map defined by extending functions to zero. 
\end{lemma}

\begin{remark}
    In \cite[Lemma~3.6]{BCF}, the authors state that $A_R(H;\Sigma|_H)$ is a subalgebra of $A_R(G;\Sigma)$. To be careful, we emphasize that there is technically an inclusion map here.  So $A_R(H;\Sigma|H)$ can be identified with a subalgebra of $A_R(G;\Sigma)$.
\end{remark}
\begin{remark}
    Since $G$ is an ample Hausdorff groupoid, inversion is open and the composition map is a homeomorphism on products of composable bisections, that is, an open map as well. Then 
    the subset $\text{Iso}(G)^\circ$ of $G$ is a subgroupoid of $G$ and is clearly open. Since $G^{(0)} \subseteq \text{Iso}(G)$ and $G^{(0)}$ is an open subgroupoid of $G$, $(\text{Iso}(G)^\circ)^{(0)}=G^{(0)}.$ We implicitly use this fact in the following result.
\end{remark}
\begin{cor}\label{cor:intisotwist}
    Let $G$ be an ample Hausdorff groupoid, let $R$ be a unital commutative ring, and let $(\Sigma,i,q)$ be a discrete twist over $G$. Then $G^{(0)} \times R^\times \to \text{Iso}(\Sigma)^\circ\to \text{Iso}(G)^\circ$ is a discrete twist and there is a natural inclusion map $\iota:A_R(\text{Iso}(G)^\circ;\text{Iso}(\Sigma)^\circ) \to A_R(G;\Sigma).$ 
\end{cor}
\begin{proof}We show that $q^{-1}(\text{Iso(G)}^\circ)=\text{Iso}(\Sigma)^\circ$. Since $q$ is a homeomorphism on unit spaces, $\sigma \in \text{Iso}(\Sigma)$ if and only if $q(\sigma) \in \text{Iso(G)}$. Since $q$ is an open map  by \cite[Proposition~2.2]{ACCCLMRSS} and $\text{Iso}(G)^\circ$ is open, we have that $q^{-1}(\text{Iso}(G)^\circ)$ is open and $q^{-1}(\text{Iso}(G)^\circ) \subseteq \text{Iso}(\Sigma)^\circ$. Similarly, since $\text{Iso}(\Sigma)^\circ$ is open and $q$ is an open map, $\text{Iso}(\Sigma)^\circ \subseteq q^{-1}(\text{Iso}(G)^\circ$. Thus the result follows from \cref{lem:BCF}.
\end{proof}

The next lemma is a key technical step in the proof of the generalised uniqueness theorem. It provides a localisation property that allows a nonzero element in $A_R(G;\Sigma)$ to be compressed to a nonzero element supported entirely on the isotropy interior. Although the argument is inspired by  \cite[Lemma~3.3]{GUTSA}, we verify here that the compression technique is carried over to the discrete twist setting.
\begin{lemma}\label{afterdense}
    Let $G$ be an ample Hausdorff groupoid, let $R$ be a unital commutative ring, and let $(\Sigma,i,q)$ be a discrete twist over $G$. Suppose $u \in \Sigma^{(0)}$ is such that $\Sigma_u^u \subseteq \text{Iso}(\Sigma)^\circ$ and take $f \in A_R(G;\Sigma)$ such that there exists $\gamma_u \in \Sigma_u^u$ with $f(\gamma_u)\not=0$. Then there exists a compact open set $K \subseteq \Sigma^{(0)}$ such that $u \in K$ and $$0 \not = \tilde{1}_K f \tilde{1}_K \in \iota(A_R(\text{Iso}(G)^\circ; \text{Iso}(\Sigma)^\circ)$$ where $\iota$ is the natural inclusion of $A_R(\text{Iso}(G)^\circ; \text{Iso}(\Sigma)^\circ)$ in $A_R(G;\Sigma)$.
    \end{lemma}

\begin{proof}
  Take $f = \sum_{U \in \mathcal{F}}r_U \tilde{1}_U \in A_R(G;\Sigma)$ where $\mathcal{F}$ is a finite collection of compact open bisections of $\Sigma$ with mutually disjoint images in $G$. Fix $u \in \Sigma^{(0)}$ such that $\Sigma^u_u \subseteq \text{Iso}(\Sigma)^\circ$. We choose a compact open neighborhood $V_D \subseteq \Sigma^{(0)}$ for each $D \in \mathcal{F}$ as follows:
   \begin{enumerate}
   \item  [Case 1:] If $r(\gamma)=u=s(\gamma)$ for some $\gamma \in D$, then $\gamma \in \text{Iso}(\Sigma)^\circ$. Now choose $V_D \subseteq \Sigma^{(0)}$ such that $u \in V_D$ and $V_D$ is a subset of the open set $r(D \cap \text{Iso}(\Sigma)^\circ) = s(D \cap \text{Iso}(\Sigma)^\circ)$. Then, $V_DDV_D \subseteq D \cap \text{Iso}(\Sigma)^\circ$.
   
   \item [Case 2:] If $r(\gamma)=u \not= s(\gamma)$ for some $\gamma \in D$, since $\Sigma^{(0)}$ is Hausdorff, we can find a compact open subset $D' \subseteq D$ containing $\gamma$ such that $r(D') \cap s(D') = \varnothing$.  Choose $V_D = s(D')$ to get $V_DDV_D = \varnothing$.
   
   \item [Case 3:] If $r(\gamma) \not= u$ for all  $\gamma \in D$, then $u \notin r(D)$ and we can pick  $V_D$ containing $u$ such that $V_D \cap r(D) = \varnothing$ and hence  $V_DDV_D = \varnothing$.
\end{enumerate}
Let $K=\bigcap_{D \in \mathcal{F}}V_D$ which is a compact open subset of $\Sigma^{(0)}$ containing $u$ by our construction of $V_D$. Hence, $KDK \subseteq V_DDV_D \subseteq D \cap \text{Iso}(\Sigma)^\circ \not= \varnothing$ and $\tilde{1}_Kf\tilde{1}_K$ is supported in $\text{Iso}(\Sigma)^\circ$, that is, $$\tilde{1}_Kf\tilde{1}_K \in \iota(A_R(\text{Iso}(G)^\circ;\text{Iso}(\Sigma)^\circ)$$
and $$(\tilde{1}_Kf \tilde{1}_K )(\gamma_u)= \tilde{1}_K(r(\gamma_u))f (\gamma_u)\tilde{1}_K(s(\gamma_u))=f(\gamma_u)\not=0$$ by assumption.
\end{proof}
We now have the machinery we need to prove the generalised uniqueness theorem for twisted Steinberg algebras. We show that injectivity of homomorphism is completely determined by its behaviour on the interior of the isotropy, provided that units with interior isotropy are dense in the unit space. This density condition replaces the ``second-countability" and ``effective" groupoid requirements commonly imposed in existing literature, thereby yielding a sharper and more flexible uniqueness criterion. This is an algebraic analogue of \cite[Theorem~6.3]{Armstrong2022} for twisted groupoid $C^*$-algebras and a generalisation of \cite[Theorem~3.1]{GUTSA}.
\begin{thm}{(Generalised Uniqueness Theorem)}\label{GUT}
    Let $G$ be an ample Hausdorff groupoid, let $R$ be a unital commutative ring, and let $(\Sigma,i,q)$ be a discrete twist over $G$. Suppose that $$X_\Sigma=\{u \in \Sigma^{(0)}: \Sigma^u_u \subseteq \text{Iso}(\Sigma)^\circ\}$$ is dense in $\Sigma^{(0)}$. Let $Q$ be a ring and $\pi: A_R(G;\Sigma) \to Q$ be a ring homomorphism. Then $\pi$ is injective if and only if for every nonzero
    $$f \in A_R(\text{Iso}(G)^\circ; \text{Iso}(\Sigma)^\circ)\quad \text{we have}\quad \pi(\iota(f)) \not = 0$$ where $\iota$ is the natural inclusion map of $A_R(\text{Iso}(G)^\circ; \text{Iso}(\Sigma)^\circ)$ on $A_R(G;\Sigma).$ 
\end{thm}
\begin{proof}
    If $\pi$ is injective, then $\pi \circ \iota$ is injective. If $\pi$ is not injective, then there exists a nonzero $g \in \text{ker}~\pi$. By \cite[Proposition~2.8]{ACCCLMRSS}, we can write $g=\sum_{\substack{D \in \mathcal{F}}}a_D \tilde{1}_D$ where $\mathcal{F}$ is a finite collection of compact open bisections in $\Sigma$ with mutually disjoint images in $G$. Pick $D_0 \in \mathcal{F}$ to define $f=\tilde{1}_{D_0^{-1}}  g$.  Then $f \in \text{ker}~\pi$. Using Lemma \ref{lemma1}, we see that for every $\alpha \in s(D_0), f(\alpha) =a_{D_0}$, that is, $f$ is nonzero and hence $\text{supp}(f) \cap \Sigma^{(0)} \not= \varnothing$.  Since $\Sigma$ is Hausdorff whenever $G$ is Hausdorff by \cite[Corollary~2.3]{ACCCLMRSS}, and $f$ is locally constant and compactly supported, $\text{supp}(f) \cap \Sigma^{(0)}$ is an open subset of $\Sigma^{(0)}$. By density of $X_\Sigma$, we may choose $u \in \text{supp}(f) \cap \Sigma^{(0)}$ where $\Sigma_u^u \subseteq \text{Iso}(\Sigma)^\circ$. Taking $\gamma_u=u$ and applying Lemma \ref{afterdense}, there exists $K \subseteq \Sigma^{(0)}$ such that $$0 \not = \tilde{1}_K f \tilde{1}_K \in \iota(A_R(\text{Iso}(G)^\circ; \text{Iso}(\Sigma)^\circ).$$Thus, $\text{supp}(\tilde{1}_K f \tilde{1}_K) \subseteq \text{Iso}(\Sigma)^\circ$. Since $A_R(\text{Iso}(G)^\circ;\text{Iso}(\Sigma)^\circ) \subseteq A_R(G;\Sigma)$ and $\text{ker}~\pi$ is an ideal, $$0 \not= \tilde{1}_K f \tilde{1}_K \in \text{ker}~(\pi \circ \iota).$$
\end{proof}
The proof of Theorem \ref{GUT} relies on density of units whose isotropy lies inside the interior of the isotropy. The next lemma guarantees this density in the second-countable setting and generalizes \cite[Lemma~3.3(a)]{BNSW2016} to discrete twists. Since the argument depends on the Baire Category Theorem, we need the unit space to be a Baire space so we add the hypothesis that $G$ is second-countable however $\Sigma$ need not be second-countable.  The result basically follows from the untwisted version \cite[Lemma~3.3(a)]{BNSW2016}.

\begin{lemma}\label{dense}
   Let $G$ be a second-countable ample Hausdorff groupoid, let $R$ be a unital commutative ring, and let $(\Sigma,i,q)$ be a discrete twist over $G$. Then $$X_\Sigma =\{u \in \Sigma^{(0)}: \Sigma_u^u \subseteq \text{Iso}(\Sigma)^\circ\}$$ is dense in $\Sigma^{(0)}$.
\end{lemma}

\begin{proof}
First, using that $q$ is a continuous open surjection that preserves source and range, we have  
$q(\text{Iso}(\Sigma)^\circ) = \text{Iso}(G)^\circ$.  We also have that for any $u \in \Sigma^{(0)}$, $q(\Sigma_u^u) = G^u_u$ and hence
    $q(u) \in X_G$ if and only if $u \in X_\Sigma$. 
    Thus, $X_\Sigma =q^{-1}(X_G) \cap \Sigma^{(0)}$. By \cite[Lemma~3.3(a)]{BNSW2016}, $q^{-1}(X_G)$ is dense in $\Sigma^{(0)}$ and so is $X_\Sigma$. 
\end{proof}

The following is immediate from \cref{GradedUT} and \cref{dense}.
\begin{cor}
Let $G$ be a second-countable ample Hausdorff groupoid, let $R$ be a unital commutative ring, and let $(\Sigma,i,q)$ be a discrete twist over $G$. Let $Q$ be a ring and $\pi: A_R(G;\Sigma) \to Q$ be a ring homomorphism. Then $\pi$ is injective if and only if for every nonzero $f \in A_R(\text{Iso}(G)^\circ; \text{Iso}(\Sigma)^\circ)~\text{we have}~ \pi(\iota(f)) \not = 0$ where $\iota$ is the natural inclusion map of $A_R(\text{Iso}(G)^\circ; \text{Iso}(\Sigma)^\circ)$ on $A_R(G;\Sigma).$     
\end{cor}

In the untwisted setting, the Cuntz-Krieger uniqueness theorem for $A_R(G)$ (see \cite[Corollary~3.4]{GUTSA}) is a direct consequence of the generalised uniqueness theorem for $A_R(G)$ (see \cite[Theorem~3.1]{GUTSA}). In Corollary \ref{CK1}, we show that the Cuntz-Krieger uniqueness theorem for $A_R(G;\Sigma)$ likewise follows from Theorem \ref{GUT}.
\begin{remark}\label{3.9}
    For an effective ample Hausdorff groupoid $G$, $\text{Iso}(G)^\circ = G^{(0)}$ which implies $\text{Iso}(\Sigma)^\circ=\Sigma^{(0)}$ since $q$ restricts homeomorphism of unit spaces. Hence, to show injectivity of a homomorphism in the twisted Steinberg algebra setting for effective groupoids, we only need to check if it is injective in $A_R(G^{(0)};\Sigma^{(0)})$ and apply Theorem \ref{GUT}. In the case of a discrete twist, it is sufficient to check injectivity in $A_R(\Sigma^{(0)})$ since the restriction of $\Sigma$ to its unit space is trivial and will give us the $R$-algebra isomorphism $A_R(G^{(0)};\Sigma^{(0)}) \cong A_R(\Sigma^{(0)})$.
\end{remark}
 \begin{cor}\label{CK1}{(Cuntz-Krieger uniqueness theorem)}
    Let $G$ be an effective ample Hausdorff groupoid, let $R$ be a unital commutative ring, and let $(\Sigma,i,q)$ be a discrete twist over $G$. Suppose that $Q$ is a ring and $\pi: A_{R}(G, \Sigma) \to Q$ is a ring homomorphism. Then $\pi$ is injective if and only if $\pi(r\Tilde{1}_W) \not= 0$ for every nonempty compact open subset $W$ of $\Sigma^{(0)}$ and $r \in R \setminus \{0\}$.
\end{cor}

\begin{proof}
By Corollary \ref{3.9} and Remark \ref{cor:intisotwist}, since \(G\) is effective,
 $\operatorname{Iso}(G)^\circ=G^{(0)}$
and hence
$\operatorname{Iso}(\Sigma)^\circ=\Sigma^{(0)}.$
Thus
$$ A_R(\operatorname{Iso}(G)^\circ; \operatorname{Iso}(\Sigma)^\circ) \cong A_R(\Sigma^{(0)}). $$
Suppose first that \(\pi\) is injective. Then clearly
$ \pi(r\widetilde{1}_W)\neq0$
for every \(r\in R\setminus\{0\}\) and every nonempty compact open \(W\subseteq\Sigma^{(0)}\).
For the converse, we first show that
$ \pi|_{A_R(\Sigma^{(0)})} $
is injective. Let \(0\neq h\in A_R(\Sigma^{(0)})\). Choose \(u\in\Sigma^{(0)}\) with \(h(u)=r\neq0\). Since \(h\) is locally constant, there exists a nonempty compact open neighborhood \(W\) of \(u\) such that
$ h|_W=r.$
Hence
$\widetilde{1}_W*h=r\widetilde{1}_W\neq0.$
If \(\pi(h)=0\), then
$$ \pi(r\widetilde{1}_W) = \pi(\widetilde{1}_W*h) = \pi(\widetilde{1}_W)\pi(h) =0, $$
contradicting the hypothesis. Thus \(\pi|_{A_R(\Sigma^{(0)})}\) is injective.
Since \(G\) is effective, \(X_\Sigma=\Sigma^{(0)}\), and Theorem 3.6 now implies that \(\pi\) is injective.  
\end{proof}

We next recover \cite[Theorem~6.1]{ACCLMR} as a corollary of the generalized uniqueness theorem. Using the key observation that any locally constant $2$-cocycle $\sigma$ on a Hausdorff \'etale  groupoid $G$ gives rise to a discrete twist over $G$ (see \cite[Example~4.5]{ACCLMR}), we reduce the problem to an application of Corollary \ref{CK1} and obtain an injectivity criterion in terms of compact open subsets of the unit space.
\begin{cor}\label{sigma}
    Let $G$ be an  effective ample Hausdorff groupoid, let $R$ be a field and $\sigma:G^{(2)} \to R^{\times}$ be a continuous 2-cocycle. Suppose that $Q$ is a ring and that $\pi:A_R(G;\sigma) \to Q$ is a ring homomorphism. Then $\pi$ is injective if and only if $\pi(1_W)\not=0$ for every nonempty compact open subset $W$ of $G^{(0)}$.
\end{cor}

\section{Generalised Graded Uniqueness Theorem}
\label{sec4}
    In this section, we prove a generalised graded uniqueness theorem for twisted Steinberg algebras where $(\Sigma, G, \Gamma)$ is a graded discrete $R$-twist as in \cite{naingue}, and $\Gamma$ is a discrete group with identity $\epsilon$. This result extends \cite[Theorem~3.4]{CE2015} by eliminating the ``effective" hypothesis on the homogeneous component with degree $\epsilon$. As a consequence, we recover \cite[Theorem~7.2]{ACCCLMRSS} and extend it to the discrete twist setting as well.

    For the proof of the generalised graded uniqueness theorem, it is sufficient to check injectivity on $A_R(G;\Sigma)_\epsilon$ and apply Theorem \ref{GUT}. To do that, we need the following result.
    \begin{lemma}\label{4.1}
        Let $G$ be an ample Hausdorff groupoid, let $R$ be a unital commutative ring, let $\Gamma$ be a discrete group with identity $\epsilon$. Suppose that $(\Sigma, G, \Gamma)$ is a graded discrete $R$-twist over $G$ and $c_G: G \to \Gamma$ and $c_\Sigma: \Sigma \to \Gamma$ be $\Gamma$-gradings on $G$ and $\Sigma$, respectively. Then, $$G_\epsilon^{(0)} \times R^\times \overset{i}{\hookrightarrow} \Sigma_\epsilon \overset{q}{\hookrightarrow} G_\epsilon$$ is a discrete $R$-twist over $G_\epsilon$ and the natural inclusions $G_\epsilon\to G$ and $\Sigma_\epsilon \to \Sigma$ induces a surjective algebra homomorphism $\iota_\epsilon: A_R(G_\epsilon;\Sigma_\epsilon) \to A_R(G;\Sigma)_\epsilon$. Moreover, $G_\epsilon^{(0)} =G^{(0)}$ and $\Sigma_\epsilon^{(0)} = \Sigma^{(0)}$.
    \end{lemma}
\begin{proof}
    Since $(\Sigma,G,\Gamma)$ is a graded discrete $R$-twist, $c_\Sigma=q \circ c_G$. Hence, $$q^{-1}(G_\epsilon)=q^{-1}(c^{-1}_G(\epsilon)) = c^{-1}_\Sigma(\epsilon)=\Sigma_\epsilon.$$
Since $G_\epsilon$ is a clopen subgroupoid of $G$ and $q$ is a homeomorphism, $\Sigma_\epsilon$ is clopen in $\Sigma$. Thus, our desired result follows from Lemma \ref{lem:BCF}. Since the grading maps are groupoid homomorphisms, they preserve
units, so
\[
G_\epsilon^{(0)}=G^{(0)},\qquad
\Sigma_\epsilon^{(0)}=\Sigma^{(0)}.
\]
Finally, extending functions by zero outside $\Sigma_\epsilon$
induces the isomorphism $\iota_\epsilon$, since
$A_R(G;\Sigma)_\epsilon$ consists precisely of functions
supported in $\Sigma_\epsilon$.
\end{proof}

    One key technical tool is the observation that any nonzero homogeneous element may be reduced via convolution with a suitably chosen compactly supported functions, to an element supported in the homogeneous component of degree $\epsilon$. The next lemma establishes this reduction which we use in the proof of our main theorem.
\begin{lemma}\label{gradedlemma}
        Let $G$ be an ample Hausdorff groupoid, let $R$ be a unital commutative ring, let $\Gamma$ be a discrete group with identity $\epsilon$. Suppose that $(\Sigma,G, \Gamma)$ is a graded discrete $R$-twist where $c_G:G \to \Gamma$ and $c_\Sigma: \Sigma \to \Gamma$ are $\Gamma$-gradings on $G$ and $\Sigma$, respectively. Then for each nonzero $g_\gamma \in A_R(G;\Sigma)_\gamma$, there exists a compact open bisection $B$ such that $$f= \tilde{1}_{B^{-1}} *g_\gamma \in \iota_\epsilon(A_R(G_\epsilon;\Sigma_\epsilon))$$ and $\text{supp}(f) \cap \Sigma^{(0)} \not= \varnothing$ where $\iota_\epsilon: A_R(G_\epsilon;\Sigma_\epsilon) \to A_R(G;\Sigma)_\epsilon$ is the natural inclusion map.
    \end{lemma}
    \begin{proof}
       For a nonzero $g_\gamma \in A_R(G;\Sigma)_\gamma$ with $\gamma \in \Gamma$, \cite[Lemma~4.3(a)]{naingue} implies $g_\gamma = \sum_{\substack{U \in \mathcal{F}}}r_U\tilde{1}_U$ where
       $\mathcal{F}$ is a finite set of $\gamma$-homogeneous compact open bisections of $\Sigma$ with mutually disjoint images in $G$ and each $r_U \in R$ is nonzero. Fix a compact open bisection $B \in \mathcal{F}$ and let $f= \tilde{1}_{B^{-1}}*g_\gamma$. Then we have \[
       f= r_B\tilde{1}_{B^{-1}B} + \sum_{\substack{U\in \mathcal{F} \\ U \not= B}} r_U \tilde{1}_{B^{-1}U}.\]
       To show that $\text{supp}(f) \cap \Sigma^{(0)}\not= \varnothing$, pick an element $x \in B^{-1}B =s(B) \subseteq \Sigma^{(0)}.$ By Lemma \ref{lemma1}, $f(x)= r_B \not= 0$ so $x \in \text{supp}(f) \cap \Sigma^{(0)}$. Now since $A_R(G;\Sigma)$ is graded and $g$ is homogeneous, $\text{supp}(f) \subseteq \Sigma_{\gamma^{-1}\gamma} \subseteq \Sigma_\epsilon$. Hence, $f \in \iota_\epsilon(A_R(G_\epsilon;\Sigma_\epsilon))$ and there exists $f_\epsilon \in A_R(G_\epsilon;\Sigma_\epsilon)$ such that $f=\iota_\epsilon(f_\epsilon)$.\end{proof}

\begin{thm}{(Generalised Graded Uniqueness Theorem)} \label{GGUT}
    Let $G$ be an ample Hausdorff groupoid, let $R$ be a unital commutative ring, let $\Gamma$ be a discrete group with identity $\epsilon$, and let $(\Sigma,G, \Gamma)$ be a graded disrete $R$-twist. Let $Q$ be a graded ring and $\pi: A_R(G;\Sigma) \to Q$ be a  graded ring homomorphism. Suppose that $X_{\Sigma_\epsilon}\coloneqq \{u \in \Sigma_\epsilon^{(0)}: (\Sigma_\epsilon)^u_u \subseteq \text{Iso}(\Sigma_\epsilon)^\circ\}$ is dense in $\Sigma^{(0)}$. Then $\pi~\text{is injective}$ if and only if  for every nonzero $$h \in A_R(\text{Iso}(G_\epsilon)^\circ;\text{Iso}(\Sigma_\epsilon)^\circ)\quad \text{ we have } \quad \pi(\iota_\epsilon(\iota(h))) \not= 0$$  where $\iota_\epsilon:A_R(G_\epsilon;\Sigma_\epsilon) \to A_R(G;\Sigma)_\epsilon$ and $\iota: A_R(\text{Iso}(G_\epsilon)^\circ;\text{Iso}(\Sigma_\epsilon)^\circ) \to A_R(G_\epsilon;\Sigma_\epsilon)$ are the inclusion maps from Lemma \ref{4.1} and Corollary \ref{cor:intisotwist} respectively.
\end{thm}
 \begin{proof}
      The sufficient part is trivial since if $\pi$ is injective, $\iota_\epsilon$ and $\iota$ are injective as natural inclusion maps. Suppose $\pi$ is not injective and $g \in \text{ker}~\pi$ where $g$ is a nonzero element of $A_R(G;\Sigma)$. Since $A_R(G;\Sigma)$ is graded, we can write $g = \sum_{\substack{\gamma \in F}} g_\gamma$ where $F$ is a finite subset of $\Gamma$ and $g_\gamma \in A_R(G;\Sigma)_\gamma$ for each $\gamma \in \Gamma$. Since $\pi$ is graded and graded components in $Q$ are linearly independent, $\sum_{\gamma \in F}\pi(g_\gamma)=0$ implies $\pi(g_\gamma)=0$ for each $\gamma \in \Gamma$. Fix $\gamma \in \Gamma$ and $a \in \Sigma_\gamma$ such that $g_\gamma(a) \not=0$.
        By Lemma~4.2, 
        there exists a compact open bisection $B$ and nonzero $f:=\widetilde{1}_{B^{-1}}*g_\gamma
\in \iota_\epsilon(A_R(G_\epsilon;\Sigma_\epsilon))$ where $\iota_\epsilon(f_\epsilon)=f$ for some $f_\epsilon \in A_R(G_\epsilon;\Sigma_\epsilon)$. Since $g_\gamma\in\ker(\pi)$, we obtain
\[
(\pi\circ\iota_\epsilon)(f_\epsilon)
=
\pi\bigl(\iota_\epsilon(f_\epsilon)\bigr)
=\pi(f)=
\pi\bigl(\widetilde{1}_{B^{-1}}*g_\gamma\bigr)
=
\pi(\widetilde{1}_{B^{-1}})
\pi(g_\gamma)
=0.
\]
Thus $\pi\circ\iota_\epsilon$ is not injective. Since we assume $X_{\Sigma_\epsilon}$ to be dense in $\Sigma^{(0)}$, we can apply Theorem 3.6 to get nonzero $h \in A_R(\text{Iso}(G_\epsilon)^\circ;\text{Iso}(\Sigma_\epsilon)^\circ)$ such that $\pi(\iota_\epsilon(\iota(h)))=0$.
 \end{proof}  
 
The next result is an extension of \cite[Theorem~7.2]{ACCLMR} to a discrete twist setting. We show that the uniqueness theorem for $A_R(G;\Sigma)$ where the homogeneous component with degree $\epsilon$ is effective follows from Theorem \ref{GGUT}.

\begin{cor}{(Graded Uniqueness Theorem)}\label{GradedUT}
     Let $G$ be an ample Hausdorff groupoid, let $R$ be a unital commutative ring  and let $\Gamma$ be a discrete group with identity $\epsilon$. Suppose that $(\Sigma, G, \Gamma)$ is a graded discrete $R$-twist and the subgroupoid $G_\epsilon$ is effective. If $Q$ is a $\Gamma$-graded ring and $\pi:A_R(G;\Sigma) \to Q$ is a graded ring homomorphism, then $\pi$ is injective if and only if $\pi(r\tilde{1}_K) \not=0$ for every non-empty compact open subset $K$ of $\Sigma^{(0)}$ and $r \in R$.
\end{cor}
\begin{proof}
Since $\text{Iso}(G_\epsilon)^\circ=G^{(0)}$ by definition and $q$ is a local homeomorphism restricting to homeomorphism of unit spaces,  $\text{Iso}(\Sigma_\epsilon)^\circ= \Sigma^{(0)}.$ Hence, $$A_R(\text{Iso}(G_\epsilon)^\circ;\text{Iso}(\Sigma_\epsilon)^\circ) = A_R(G^{(0)};\Sigma^{(0)}) \cong A_R(\Sigma^{(0)}).$$ 
By Remark \ref{3.9}, the density of $X_{\Sigma_\epsilon}$ in $\Sigma^{(0)}$ is guaranteed. Therefore, the density requirement of Theorem \ref{GGUT} is satisfied and applying it yields the desired result.
    \end{proof}

    Lastly, we recover \cite[Theorem~7.2]{ACCLMR} as a consequence of Corollary \ref{GradedUT} by virtue of \cite[Example~4.5]{ACCLMR}.

\begin{cor}
    Let $G$ be an ample Hausdorff groupoid, let $R$ be a field, and let $\sigma: G^{(2)} \to R^\times$ be a continuous $2$-cocycle. Let $\Gamma$ be a discrete group with identity $\epsilon$, and suppose that $c: G \to \Gamma$ is a continuous groupoid homomorphism such that the subgroupoid $G_\epsilon$ is effective. Suppose $Q$ is a $\Gamma$-graded ring and that $\pi: A_R(G;\sigma) \to Q$ is a graded ring homomorphism. Then $\pi$ is injecive if and only if $\pi(1_K)\not=0$ for every non-empty compact open subset $K$ of $G^{(0)}$.
\end{cor}

\noindent\textbf{Acknowledgements} This research is part of the PhD thesis of the first-named author. The first-named author thanks the third-named author and Te Herenga Waka, Victoria University of Wellington for their hospitality and for appointing her as a research visiting scholar.

\vspace{0.5em}

\noindent\textbf{Corresponding Author} Rizalyn Bongcawel 

\vspace{0.5em}


\vspace{0.5em}

\noindent\textbf{Data Availability} No datasets were generated or analysed during the current study.  

\vspace{0.5em}

\noindent\textbf{Competing Interest} The authors declare no competing interests.

\vspace{0.5em}

\noindent\textbf{Funding} The first-named author was funded by Department of Science and Techonology (Accelerated Science and Technology Human Resource and Development Program), Philippines and Te Herenga Waka - Victoria University of Wellington, New Zealand. The second-named author was supported by Office of the Vice Chancellor for Research and Extension of MSU-IIT through SO 00464-IIT, S. 2026.
\bibliographystyle{amsplain}
\bibliography{references}

\end{document}